\documentclass{article}
\usepackage[margin=2.8cm]{geometry} 
\usepackage{graphicx} 

\usepackage{amsfonts}
\usepackage{amsthm}
\usepackage{amsmath}
\usepackage{amssymb}
\usepackage{amscd} 
\usepackage{enumitem}
\usepackage{algorithm2e}
\RestyleAlgo{ruled}
\usepackage{mathtools}
\usepackage[normalem]{ulem}
\usepackage[dvipsnames]{xcolor}
\usepackage{hyperref}
\usepackage{comment}

\newtheorem{theorem}{Theorem}[section]
\newtheorem{lemma}[theorem]{Lemma}

\newtheorem*{claim*}{Claim}

\newtheorem{proposition}[theorem]{Proposition}

\newcommand{\restatedtheoremname}{}
\newtheorem*{restatedtheoreminner}{\restatedtheoremname}
\newenvironment{restatedtheorem}[1]
  {\renewcommand{\restatedtheoremname}{Theorem~\ref{#1}}%
   \begin{restatedtheoreminner}}
  {\end{restatedtheoreminner}}

\theoremstyle{definition}
\newtheorem{definition}[theorem]{Definition}

\theoremstyle{remark}
\newtheorem{remark}[theorem]{Remark}

\newcommand{\cl}{\mathrm{cl}}

\title{On the enumeration of polymatroids}
\author{Seonghyuk Im\thanks{Center for AI and Natural Sciences, Korea Institute for Advanced Study (KIAS), Seoul, South Korea. Email: \texttt{
seonghyuk@kias.re.kr}.} \and Donggyu Kim\thanks{Department of Mathematics, Hanyang University, Seoul, South Korea. Email: \texttt{donggyu@hanyang.ac.kr}.}}
\date{\today}

\begin{document}

\maketitle

\begin{abstract}
    Let $p_k(n)$ be the number of $k$-polymatroids on $[n]$. We show that for every fixed $k \geq 1$, we have
    \[
        \left\lfloor \frac{k}{2} \right\rfloor \cdot \binom{n}{\lfloor n/2 \rfloor} \cdot (1+o(1)) \le \log_2 p_k(n) \le k \cdot \binom{n}{\lfloor n/2 \rfloor} \cdot (1+o(1)).
    \]
    We also show that for $k \geq 2$, almost all $k$-polymatroids are (i) connected, (ii) proper, and (iii) not linearly representable over any field.
\end{abstract}

\section{Introduction}

Enumerating matroids on a fixed ground set is a longstanding problem in matroid theory. Let $m_n$ denote the number of matroids on $[n]:=\{1,2,\dots,n\}$.
Following several earlier results~\cite{Crapo1965,Bollobas1969,PW1971}, Knuth~\cite{Knuth1974} established the lower bound
\[
    \log_2 \log_2 m_n \ge n - \frac{3}{2} \log_2 n - O(1),
\]
which remains the best known. For the upper bound, Piff~\cite{Piff1973} showed that
\[
    \log_2 \log_2 m_n \le n - \log_2 n + \log_2\log_2 n + O(1),
\]
leaving a gap of order $\log_2 n$. Bansal, Pendavingh, and van der Pol~\cite{BPP2015} later reduced this gap to $O(1)$ by proving that
\[
    \log_2 \log_2 m_n \leq n - \frac{3}{2} \log_2 n + O(1).
\]
The constant-order term, however, remains undetermined.

The enumeration of matroids with prescribed properties has also received considerable attention. For example, Nelson~\cite{Nelson2018} proved that the number of matroids on $[n]$ that are representable over at least one field is at most single-exponential in $n$. It follows that almost all matroids are not representable over any field. Lowrance, Oxley, Semple, and Welsh~\cite{LOSW2013} showed that almost all matroids are connected, confirming a conjecture of Mayhew, Newman, Welsh, and Whittle~\cite{MNWW2011}. See also~\cite{BPP2015,PvdP2017,PvdP2019} for related enumeration results.

Polymatroids, introduced by Edmonds~\cite{Edmonds1970} in 1970, generalize matroid rank functions. They connect discrete and convex structures and arise in several areas, including combinatorial optimization~\cite{Edmonds1970}, information theory~\cite{Fujishige1978}, cryptography~\cite{FMP2011}, and commutative algebra~\cite{HH2002}. We refer the reader to~\cite[\S1.1]{Murota2003} for further motivation and historical background.

An \emph{integral polymatroid} on a finite set $E$ is a function $f: 2^E \to \mathbb{Z}_{\ge0}$ satisfying the following properties:
\begin{enumerate}[label=(\roman*)]
    \item (Normalization) $f(\emptyset) = 0$;
    \item (Monotonicity) $f(A) \le f(B)$ for $A \subseteq B \subseteq E$;
    \item (Submodularity) $f(A) + f(B) \ge f(A \cup B) + f(A \cap B)$ for $A, B \subseteq E$.
    \end{enumerate}
Throughout this paper, we omit the qualifier ``integral'' and simply use \emph{polymatroid} to mean \emph{integral polymatroid}.
For a positive integer $k$, a polymatroid $f$ is a \emph{$k$-polymatroid} if $f(A) \le k|A|$ for all $A \subseteq E$. Thus, $1$-polymatroids are precisely matroid rank functions.

It is natural to ask whether enumeration results for matroids extend to polymatroids. To the best of our knowledge, the first enumeration result for polymatroids was obtained by Savitsky~\cite{Savitsky2014}. Welsh's conjecture~\cite{Welsh1971} asserts that the number of matroids on $[n]$ is unimodal in rank. Savitsky disproved the analogous statement for $2$-polymatroids by computationally constructing a catalog of all $2$-polymatroids on ground sets of size at most $7$.

In this paper, we obtain asymptotic upper and lower bounds on the number of $k$-polymatroids for every fixed positive integer $k$.
Let $N_n := \binom{n}{\lfloor n/2 \rfloor}$ denote the size of the largest level of the Boolean lattice $2^{[n]}$. Let $p_k(n)$ denote the number of $k$-polymatroids on $[n]$. Throughout the paper, all logarithms are to base $2$. Our main result is the following.

\begin{theorem}\label{thm:asymptotic_counting}
    For every fixed positive integer $k$,
    \[
     \lfloor k/2 \rfloor \cdot N_n \cdot (1+o(1)) \le \log p_k(n) \le k \cdot N_n \cdot (1+o(1)).
    \]
\end{theorem}

We note that this implies that $\log p_k(n) = \Theta(N_n)$ for all $k \geq 2$ while the above mentioned results show that $\log p_1(n) = \Theta(N_n/n)$. Thus, the behavior of $p_k(n)$ is completely different for when $k=1$ and for when $k \geq 2$.
We additionally note that the number of isomorphism classes of $k$-polymatroids on $n$ elements lies between $p_k(n)/n!$ and $p_k(n)$. Since $\log(n!) = o(N_n)$, Theorem~\ref{thm:asymptotic_counting} yields the same asymptotic bounds for the logarithm of this number.   

We also prove that almost all $k$-polymatroids are connected and not linearly representable over any field, extending the corresponding results for matroids~\cite{LOSW2013,Nelson2018}. In addition, for $k\ge 2$, we show that almost all $k$-polymatroids are proper, meaning that they are essentially distinct from matroids.

A polymatroid $f$ on $E$ is \emph{disconnected} if it can be expressed as the sum of two nontrivial polymatroids on disjoint sets, i.e., there exist a partition $E=E_1 \cup E_2$ into nonempty sets and polymatroids $f_1$ on $E_1$ and $f_2$ on $E_2$ such that $f(A) = f_1(A \cap E_1) + f_2(A \cap E_2)$.
We say that $f$ is \emph{connected} otherwise.

\begin{theorem}
    \label{thm-intro:connected}
    For each fixed $k \geq 1$, almost all $k$-polymatroids are connected.
\end{theorem}

For a field $\mathbb{F}$, a $k$-polymatroid $f$ on $[n]$ is \emph{$\mathbb{F}$-linearly representable}, or simply \emph{$\mathbb{F}$-representable}, if there is a collection of subspaces $V_1, V_2, \dots, V_n$ of a finite-dimensional vector space over $\mathbb{F}$ such that $f(A) = \dim \left( \sum_{i \in A} V_i \right)$ for all $A \subseteq [n]$; see~\cite{FMP2011}.
\begin{theorem}
    \label{thm-intro:representable}
    For each fixed $k \geq 1$, almost all $k$-polymatroids are not representable over any field.
\end{theorem}

For $k=1$, Theorems~\ref{thm-intro:connected} and~\ref{thm-intro:representable} follow from the corresponding results for matroids~\cite{LOSW2013,Nelson2018}. Our main contribution is to extend both results to $k \geq 2$.

Adding a nonnegative integral modular function $q$ to a matroid rank function $r_M$ yields another polymatroid $r_M+q$.  At the level of bases, this operation simply translates every basis by the same integer vector.  Following~\cite[\S4.3]{BHKKL2025}, we call a polymatroid \emph{proper} if its set of bases is not an integral translate of the set of bases of a matroid.
Equivalently, a polymatroid $f$ on $E$ is proper if and only if the translate $f-q$ is not a matroid rank function, where $q: 2^E \to \mathbb{Z}_{\ge 0}$ is the modular function defined by $q(\{i\}) = f(E) - f(E\setminus\{i\})$ for $i \in E$. The quantity $q(\{i\})$ equals the minimum $i$-th coordinate among the bases of $f$.

\begin{theorem}
    \label{thm-intro:proper}
    For each fixed $k \geq 2$, almost all $k$-polymatroids are proper.
\end{theorem}

\paragraph{Organization.}
The rest of this paper is organized as follows.
In Section~\ref{sec:prelim}, we collect definitions and basic properties of polymatroids and submodular functions. Sections~\ref{sec:upperbound} and~\ref{sec:lowerbound} establish the upper and lower bounds, respectively, in Theorem~\ref{thm:asymptotic_counting}; Section~\ref{sec:lowerbound} also proves Theorem~\ref{thm-intro:proper}. Finally, Sections~\ref{sec:connectivity} and~\ref{sec:representability} prove Theorems~\ref{thm-intro:connected} and~\ref{thm-intro:representable}, respectively.

\section{Preliminaries}\label{sec:prelim}

We first recall a local characterization of submodular functions that will be a key ingredient in the proof of Theorem~\ref{thm:asymptotic_counting}.

\begin{theorem}[see {\cite[Thm.~44.1]{Schrijver2003a}}]
    \label{thm:local-submodularity}
    A set function $f : 2^E \to \mathbb{Z}_{\ge0}$ is submodular if it satisfies the following \emph{local submodularity}: 
    \begin{center}
        $f(S\cup\{x\}) + f(S\cup\{y\}) \ge f(S\cup\{x,y\}) + f(S)$
    \end{center}
    for all $S \subseteq E$ and distinct $x,y \in E \setminus S$.
\end{theorem}

We also recall several notions for polymatroids that generalize the corresponding notions for matroids. They will be used to prove that almost all $k$-polymatroids are connected.

\begin{definition} 
    Let $f$ be a polymatroid on $E$.
    The \emph{rank} of $f$ is $f(E)$.
    The \emph{closure} of $T \subseteq E$ in $f$ is 
    \[
        \cl_f(T) := \{i \in E \mid f(T\cup\{i\}) = f(T)\}.
    \]
    We say a subset $F \subseteq E$ is a \emph{flat} if $F = \cl_f(T)$ for some $T$.
\end{definition}

\begin{definition}
    The \emph{$k$-dual} of a $k$-polymatroid $f$ on $E$ is a function $f^* : 2^{E} \to \mathbb{Z}_{\ge 0}$ defined by 
    \[
        f^*(A) 
        := k|A| - f(E) + f(E \setminus A).
    \]
    The $k$-dual $f^*$ is a $k$-polymatroid; see \cite[\S44.6f]{Schrijver2003a}.
\end{definition}

Note that the ordinary dual of a matroid is identical to the $1$-dual of the corresponding $1$-polymatroid.

\begin{definition}\label{def:principal extension}
    Let $f$ be a polymatroid on $[n]$ and $F$ be its flat.
    The \emph{principal extension} of $f$ on $F$ is the function $p_{f,F} : 2^{[n+1]} \to \mathbb{Z}_{\ge0}$ defined by
    \[
        p_{f,F}(A):=f(A),\qquad
        p_{f,F}(A\cup \{n+1\}):=\min\{f(A)+1,f(A\cup F)\}
        \quad(A\subseteq[n]).
    \]
\end{definition}

\begin{proposition}[\cite{Lovasz1983}]\label{prop:principal extension}
    Every principal extension of a polymatroid is a polymatroid.
\end{proposition}

\section{Upper bound}\label{sec:upperbound}

In this section, we prove the upper bound in
Theorem~\ref{thm:asymptotic_counting} by exposing one element at a time. A high-level sketch is the following: A $k$-polymatroid on $[n]$ is determined by its restriction to $[n-1]$ together with its discrete derivative in the direction $n$. Submodularity forces this derivative to be a decreasing $\{0,1,\ldots,k\}$-valued function, which can be encoded by $k$ monotone Boolean functions.  The asymptotic solution to the Dedekind problem then gives a recursive upper bound for $p_k(n)$, and iterating this bound yields the desired estimate.

For a $k$-polymatroid $f:2^{[n]} \to \mathbb{Z}_{\geq 0}$ and $i\in[n]$, define the \emph{discrete derivative} $\nabla_i f:2^{[n] \setminus \{i\}} \to \mathbb{Z}_{\geq 0}$ of $f$ in direction $i$ by
\[
    \nabla_i f(A) = f(A \cup \{i\}) - f(A).
\]
We begin with two basic properties of this derivative.
\begin{lemma}\label{lem:discrete derivative}
    For every $k$-polymatroid $f$ and $i \in [n]$, the function $\nabla_i f$ satisfies the following properties:
    \begin{itemize}
        \item $0 \leq \nabla_i f \leq k$;
        \item $\nabla_i f$ is monotonically decreasing; that is, if $A \subseteq B \subseteq [n] \setminus \{i\}$, then $\nabla_i f(A) \geq \nabla_i f(B)$.
    \end{itemize}
\end{lemma}
\begin{proof}
    Let $A \subseteq [n]\setminus\{i\}$. Since $f$ is monotone, $f(A\cup\{i\}) \ge f(A)$, and hence $\nabla_i f(A)=f(A\cup\{i\})-f(A) \ge 0$.
    On the other hand, submodularity applied to $A$ and $\{i\}$ gives
    \[
        f(A)+f(\{i\}) \ge f(A\cup\{i\})+f(\emptyset)=f(A\cup\{i\}),
    \]
    so
    \[
        \nabla_i f(A)=f(A\cup\{i\})-f(A) \le f(\{i\}) \le k.
    \]
    This proves $0\le \nabla_i f \le k$.

    For the second claim, it is enough to show that for every $A \subseteq [n]\setminus\{i\}$ and every $j \in [n]\setminus (A\cup\{i\})$,
    \[
        \nabla_i f(A) \ge \nabla_i f(A\cup\{j\}).
    \]
    By submodularity,
    \[
        f(A\cup\{i\}) + f(A\cup\{j\}) \ge f(A\cup\{i,j\}) + f(A),
    \]
    and rearranging gives
    \[
        \nabla_i f(A)=f(A\cup\{i\})-f(A) \ge f(A\cup\{i,j\})-f(A\cup\{j\})=\nabla_i f(A\cup\{j\}).
    \]
    Therefore, $\nabla_i f$ is monotonically decreasing.
\end{proof}

To bound the number of possible discrete derivatives, we use the following asymptotic result on the Dedekind problem.
\begin{theorem}[\cite{Kleitman1969}]\label{thm:Dedekind}
    The number of monotone Boolean functions $f: 2^{[n]} \to \{0,1\}$ is $2^{(1+o(1))\binom{n}{\lfloor n/2 \rfloor}}$.
\end{theorem}

The following recursive bound is the main step in the proof of the upper bound.
\begin{lemma}
    For every fixed $k$ and sufficiently large $n$, we have
    \[ p_k(n) \leq p_k(n-1) \cdot 2^{(1+o(1))k N_{n-1}}.\]
\end{lemma}
\begin{proof}
    For a $k$-polymatroid $f$ on $[n]$, let $f_0 := f|_{2^{[n-1]}}$ be the restriction of $f$ to $2^{[n-1]}$, and let $g := \nabla_n f$ be its discrete derivative in direction $n$.
    Then $f_0$ is a $k$-polymatroid on $[n-1]$.
    Moreover, $f$ is determined by the pair $(f_0,g)$: if $A \subseteq [n-1]$, then $f(A)=f_0(A)$, and if $A=B\cup\{n\}$ with $B \subseteq [n-1]$, then
    \[
        f(A)=f(B\cup\{n\})=f(B)+\nabla_n f(B)=f_0(B)+g(B).
    \]
    Thus, it suffices to bound the number of possible functions $g$.

    By Lemma~\ref{lem:discrete derivative}, every possible $g$ satisfies $0 \le g \le k$ and is monotonically decreasing.
    For each $i \in \{0,1,\dots,k-1\}$, define
    \[
        \varphi_i^g(A) := \mathbf{1}_{\{g(A)\le i\}} \qquad (A \subseteq [n-1]).
    \]
    Since $g$ is monotonically decreasing, each $\varphi_i^g$ is a monotone Boolean function on $2^{[n-1]}$.
    Moreover, $g$ is determined by the $k$-tuple $(\varphi_0^g,\varphi_1^g,\dots,\varphi_{k-1}^g)$.
    Hence, by Theorem~\ref{thm:Dedekind}, the number of possible choices for $g$ is at most $2^{k(1+o(1))\binom{n-1}{\lfloor (n-1)/2 \rfloor}}$.
    This proves the lemma.
\end{proof}

\begin{theorem}\label{thm:asymptotic_counting-upper}
    For every fixed $k$, we have $\log p_k(n) \le k \cdot N_n \cdot (1+o(1))$.
\end{theorem}
\begin{proof}
    Fix $\varepsilon>0$.
    By the previous lemma, there exists $n_0$ such that for every $n > n_0$,
    \[
        \log p_k(n) \le \log p_k(n-1) + k(1+\varepsilon) N_{n-1}.
    \]
    Hence, for every $n > n_0$,
    \[
        \log p_k(n) \le \log p_k(n_0) + k(1+\varepsilon)\sum_{m=n_0}^{n-1} N_m.
    \]

    We claim that
    \[
        \sum_{m=1}^{n-1} N_m = N_n(1+o(1)).
    \]
    Indeed, for $m \ge 2$,
    \[
        \frac{N_{m-1}}{N_m}=
        \begin{cases}
            \frac12 & \text{if $m$ is even},\\
            \frac{\lfloor m/2 \rfloor + 1}{m} \le \frac23 & \text{if $m$ is odd}.
        \end{cases}
    \]
    Therefore $N_{n-j} \le (2/3)^j N_n$ for all $j \ge 1$, so for every fixed $L$,
    \[
        \sum_{m=1}^{n-L-1} N_m = O\bigl((2/3)^L N_n\bigr).
    \]
    On the other hand, for each fixed $j \ge 1$, each factor in
    \[
        \frac{N_{n-j}}{N_n} = \prod_{\ell=0}^{j-1} \frac{N_{n-\ell-1}}{N_{n-\ell}}
    \]
    tends to $1/2$ as $n \to \infty$, and hence
    \[
        \frac{N_{n-j}}{N_n} = 2^{-j} + o(1).
    \]
    Thus, for every fixed $L$,
    \[
        \sum_{m=1}^{n-1} N_m = \sum_{j=1}^{L} N_{n-j} + O\bigl((2/3)^L N_n\bigr)
        = \left(\sum_{j=1}^{L} 2^{-j} + o(1) + O\bigl((2/3)^L\bigr)\right)N_n.
    \]
    Letting first $n \to \infty$ and then $L \to \infty$, we obtain
    \[
        \sum_{m=1}^{n-1} N_m = N_n(1+o(1)).
    \]

    Consequently,
    \[
        \log p_k(n) \le \log p_k(n_0) + k(1+\varepsilon)N_n(1+o(1))
        = k(1+\varepsilon)N_n(1+o(1)).
    \]
    Since $\varepsilon>0$ was arbitrary, this yields
    $\log p_k(n) \le k \cdot N_n \cdot (1+o(1))$.
\end{proof}

\section{Lower bound}\label{sec:lowerbound}

In this section, we prove the lower bound in
Theorem~\ref{thm:asymptotic_counting} by constructing a large family of $k$-polymatroids. We begin with a cardinality-based polymatroid whose local submodular inequalities hold with a margin of two on $\lfloor k/2\rfloor$ levels near the middle of the Boolean lattice. On these levels, the value at each set may be decreased independently by one without violating submodularity.  As the central levels all have size $(1+o(1))N_n$, these independent perturbations produce the required number of polymatroids.  We conclude the section by using the resulting counting bounds to show that almost all $k$-polymatroids are proper when $k\ge2$.

For a set function $f: 2^{[n]} \to \mathbb{Z}$, $S \subseteq [n]$ and $x\neq y \in [n] \setminus S$, define the \emph{local defect} $\Diamond f(S; x, y)$ by
\[
    \Diamond f (S; x, y) = f(S\cup\{x\}) + f(S\cup\{y\}) - f(S\cup\{x,y\}) - f(S).
\]
And define the \emph{defect}  $\Diamond f : \{0,1,2,\dots,n-2\} \to \mathbb{Z}$ of $f$ by
\[
    \Diamond f (m) = \min_{\substack{S\in \binom{[n]}{m} \\ x \neq y\in[n]\setminus S}} \Diamond f(S; x, y).
\]
In terms of the defect, Theorem~\ref{thm:local-submodularity} says that a set function $f : 2^E \to \mathbb{Z}$ is submodular if and only if $\Diamond f (m) \ge 0$ for all $m \in \{0,1,2,\dots,n-2\}$. We will use the following two elementary lemmas.

\begin{lemma}\label{lem:easy-construction}
    Let $g: \{0,1,\ldots,n\} \to \mathbb{Z}$ satisfy $g(0)=0$, and define
    $f: 2^{[n]} \to \mathbb{Z}$ by $f(A) := g(|A|)$.  If
    $\nabla g(m) := g(m+1) - g(m)$ is non-increasing and belongs to
    $\{0,1,\ldots,k\}$ for every $0\le m<n$, then $f$ is a
    $k$-polymatroid.
\end{lemma}
\begin{proof}
    We verify the three polymatroid axioms.
    First, $f(\emptyset) = g(0) = 0$.
    Next, $g$ is non-decreasing because $\nabla g \ge 0$, so $f$ is non-decreasing by definition.
    Finally, for $0\le m \le n-2$,
    $\Diamond f(m) = 2 g(m+1) - g(m+2) - g(m) = \nabla g(m) - \nabla g(m+1) \ge 0$ because $\nabla g$ is non-increasing. By Theorem~\ref{thm:local-submodularity}, $f$ is submodular. Therefore, $f$ is a polymatroid.

    Since $f(\{i\}) = g(1) = \Delta g(0) \le k$ for every $i\in[n]$, $f$ is a $k$-polymatroid.
\end{proof}

\begin{lemma}\label{lem:simultaneous-layer-perturbation}
    Let $f:2^{[n]}\to\mathbb Z$ be submodular, and let
    $L\subseteq\{1,2,\ldots,n-1\}$ be such that
    $\Diamond f(\ell-1)\ge 2$ for every $\ell\in L$.
    If
    $\mathcal S\subseteq
        \bigcup_{\ell\in L}\binom{[n]}{\ell }
    $,
    then $f-\mathbf 1_{\mathcal S}$ is submodular. 
\end{lemma}
\begin{proof}
    Note that $\Diamond f(m) \ge 0$ for all $m\in\{0,1,\dots,n-2\}$ by Theorem~\ref{thm:local-submodularity}.
    Fix $m\in\{0,1,\dots,n-2\}$, $S\in\binom{[n]}m$, and distinct
    $x,y\in[n]\setminus S$. Because
    $\mathcal S\subseteq \bigcup_{\ell\in L}\binom{[n]}{\ell}$, we have
    \[
        \Diamond \mathbf{1}_{\mathcal S}(S;x,y) \le
        \begin{cases}
            2 & \text{if $m+1\in L$}, \\
            0 & \text{if $m+1\notin L$}.
        \end{cases}
    \]
    Indeed, in the first case the two positive terms in the local defect
    are each at most one. In the second case, those terms vanish because
    $S\cup\{x\}$ and $S\cup\{y\}$ have cardinality $m+1\notin L$.
    It follows that
    \[
        \Diamond(f-\mathbf{1}_{\mathcal S})(S;x,y)
        =
        \Diamond f(S;x,y)-\Diamond\mathbf{1}_{\mathcal S}(S;x,y)
        \ge
        \begin{cases}
            \Diamond f(m)-2 & \text{if $m+1\in L$} \\
            \Diamond f(m) & \text{if $m+1\notin L$}
        \end{cases}
        \ge 0.
    \]
    Therefore,
    $f-\mathbf{1}_{\mathcal S}$ is submodular by
    Theorem~\ref{thm:local-submodularity}.
\end{proof}
We now prove the lower bound in the main theorem.

\begin{theorem}\label{thm:asymptotic_counting-lower}
    Let $k$ be a fixed positive integer. Then
    \[
        \log p_k(n) \ge \lfloor {k}/{2} \rfloor \cdot N_n \cdot (1+o(1)).
    \]
\end{theorem}
\begin{proof}
    If $k=1$, then the asserted lower bound is zero and there is nothing
    to prove.  We may therefore assume that $k\ge2$.
    Suppose also that $n$ is sufficiently large, and set
    $r:=\lfloor n/2\rfloor \ge 1$ and $s:=\lfloor k/2\rfloor \ge 1$. 
    
    Choose a function
    $g: \{0,1,\ldots,n\}\to\mathbb Z$ such that $g(0)=0$ and
    \[
        \nabla g(m) = g(m+1) - g(m) = \begin{cases}
            k & \text{if $0\le m<r$}, \\
            k-2(m-r+1) & \text{if $r\le m<r+s$}, \\
            0 & \text{if $m\ge r+s$},
        \end{cases}
    \]
    for $m \in \{0,1,2,\dots,n-1\}$.
    Then $f(A):=g(|A|)$ is a $k$-polymatroid by Lemma~\ref{lem:easy-construction}, and 
     \[
        \Diamond f(\ell-1) 
        =  2 g(\ell) - g(\ell+1) - g(\ell-1)
        = \nabla g(\ell-1)-\nabla g(\ell)
        \ge 2 
    \]
    for all $r \le \ell \le r+s-1$.
    
    Set $L:=\{r,r+1,\ldots,r+s-1\}$. Then for every $\mathcal S\subseteq \bigcup_{\ell\in L}\binom{[n]}{\ell}$, the function $h_{\mathcal S}:=f-\mathbf 1_{\mathcal S}$ is submodular by Lemma~\ref{lem:simultaneous-layer-perturbation}.

    We verify the remaining $k$-polymatroid axioms.
    Since $r \ge 1$, we have $\emptyset \notin \mathcal S$. Hence,
    \[
        h_{\mathcal S}(\emptyset) = f(\emptyset) = g(0) = 0.
    \]
    For $A\subsetneq B \subseteq [n]$, 
    \begin{align*}
        h_{\mathcal S}(B)-h_{\mathcal S}(A)
        &=
        \mathbf{1}_{\mathcal S}(A) - \mathbf{1}_{\mathcal S}(B) + g(|B|) - g(|A|) \\
        &=
        \mathbf{1}_{\mathcal S}(A) - \mathbf{1}_{\mathcal S}(B) + \nabla g(|B|-1) + \dots + \nabla g(|A|),
    \end{align*}
    which can be negative only when $B\in \mathcal{S}$.
    However, if $B\in \mathcal{S}$, then $|B| \le r+s-1$ and thus $\nabla g(|B|-1) \ge 1$.
    Thus, $h_{\mathcal S}(A) \le h_{\mathcal S}(B)$.
    Finally,
    $h_{\mathcal S}(A)\le f(A)\le k|A|$, so $h_{\mathcal S}$ is a
    $k$-polymatroid.

    Distinct choices of $\mathcal S$ give distinct functions
    $h_{\mathcal S}$. We have therefore constructed
    $2^{\sum_{\ell=r}^{r+s-1}\binom n\ell}$
    different $k$-polymatroids.  Since $k$ is fixed and every
    $\ell\in L$ satisfies $\ell=\lfloor n/2 \rfloor +O_k(1)$,
    \[
        \sum_{\ell=r}^{r+s-1}\binom n\ell
        =sN_n(1+o(1))
        =\lfloor k/2\rfloor N_n(1+o(1)).
    \]
    This proves the lower bound.
\end{proof}

\begin{proof}[Proof of Theorem~\ref{thm:asymptotic_counting}]
    This follows from Theorems~\ref{thm:asymptotic_counting-upper} and~\ref{thm:asymptotic_counting-lower}.
\end{proof}

As an immediate application of Theorem~\ref{thm:asymptotic_counting}, we show that almost all $k$-polymatroids are proper.


\begin{restatedtheorem}{thm-intro:proper}
    For each fixed $k \geq 2$, almost all $k$-polymatroids are proper.
\end{restatedtheorem}
\begin{proof}
    Recall that a polymatroid $f$ on $E$ is not proper if and only if $f-q$ is a $1$-polymatroid, where $q(A) := \sum_{x\in A} f(E)-f(E\setminus\{x\})$ for $A\subseteq E$.
    If $f$ is a $k$-polymatroid, then $0 \le q(\{x\}) \le k$ for each $x\in E$.
    Therefore, the number of non-proper $k$-polymatroids is bounded above by $(k+1)^n \cdot p_1(n)$.
    By \cite{BPP2015}, we have $\log p_1(n) = O(N_n/n)$. Thus, by Theorem~\ref{thm:asymptotic_counting}, we have 
    \[ n \cdot \log(k+1) + \log p_1(n) - \log p_k(n) \to -\infty,\]
    and hence
    \[
        \frac{\#\text{non-proper $k$-polymatroids on $[n]$}}{\#\text{$k$-polymatroids on $[n]$}} \to 0
    \]
    as $n \to \infty$.
\end{proof}

\section{Almost all \texorpdfstring{$k$}{k}-polymatroids are connected}\label{sec:connectivity}

We prove that almost all $k$-polymatroids are connected. The key auxiliary estimate, Lemma~\ref{lem:ground-set-growth}, is a recursive lower bound on $p_k(n)$ obtained from principal single-element extensions: duality ensures that at least half of all $k$-polymatroids have large rank, each such polymatroid has many flats, and distinct flats yield distinct principal extensions. We then bound the number of disconnected polymatroids by summing over the size of a smallest component. The extension estimate handles the case in which a polymatroid has a component of bounded size, while Theorem~\ref{thm:asymptotic_counting} shows that the remaining cases also make a negligible contribution.

\begin{lemma}\label{lem:flat-count}
    Let $f$ be a $k$-polymatroid on $E$ and $r:= f(E)$ be the rank of $f$. Then $f$ has at least $2^{\lceil r/k \rceil}$ flats.
\end{lemma}

\begin{proof}
    Choose an inclusion-minimal spanning set $A\subseteq E$, i.e., $\cl_f(A) = E$ and $\cl_f(A\setminus\{x\}) \ne E$ for all $x\in A$.  Since
    $r=f(A)\le k|A|$, we have $|A|\ge\lceil r/k\rceil$.  

    Now take distinct $I,J\subseteq A$ and, after interchanging them if
    necessary, choose $x\in I\setminus J$.  Then $x\in\cl_f(I)$.
    However, $x\notin\cl_f(J)$ because
    \[
        f(J\cup\{x\})-f(J)\ge f(A)-f(A\setminus\{x\})>0,
    \]
    where the first inequality holds by submodularity and the second inequality holds by our choice of $A$.
    Therefore, $\cl_f(I)\ne\cl_f(J)$. 
    Hence, the closures of the $2^{|A|}$ subsets of $A$ are distinct flats, which proves the claim.
\end{proof}

We next establish a recursive lower bound on the number of $k$-polymatroids. This bound, obtained from principal single-element extensions, will be used to handle disconnected polymatroids with a component of bounded size.

\begin{lemma}
\label{lem:ground-set-growth}
For every $k\ge1$ and $n\ge1$,
\begin{equation}\label{eq:ground-set-growth}
    p_k(n+1)\ge 2^{\lceil n/2\rceil-1}p_k(n).
\end{equation}
Consequently, for $1\le t<n$,
\begin{equation}\label{eq:iterated-ground-growth}
    \frac{p_k(n)}{p_k(n-t)}
    \ge 2^{\,t(2n-t-5)/4}.
\end{equation}
\end{lemma}

\begin{proof}
The $k$-duality involution pairs $k$-polymatroids on $[n]$ of rank $r$ with those of rank $kn-r$.
Hence, at least half of the $k$-polymatroids on $[n]$ have rank at least $kn/2$. By Lemma~\ref{lem:flat-count}, each of them has at least $2^{\lceil n/2\rceil}$ flats.

For each flat $F$ of such a polymatroid $f$, its principal extension $p_{f,F}$ is a polymatroid on $[n+1]$; see Definition~\ref{def:principal extension} and Proposition~\ref{prop:principal extension}.
Deleting the new element $n+1$ recovers $f$, so
extensions coming from different $f$ are distinct.
Recall that $p_{f,F}(A) = \min\{f(A)+1, f(A\cup F)\}$ for $A\subseteq[n]$.
Then
\[
 \{A:p_{f,F}(A\cup \{n+1\})=f(A)\}
 =\{A:F\subseteq\cl_f(A)\}.
\]
The intersection of the closures $\cl_f(A)$ over all sets $A$ in this family is $F$ itself.
Therefore, for a fixed $f$, the extension $p_{f,F}$ determines $F$. Thus, all the principal extensions counted above are distinct, and hence
\[
    p_k(n+1) \ge \frac{1}{2} \cdot 2^{\lceil n/2\rceil} \cdot p_k(n).
\]
This proves \eqref{eq:ground-set-growth}.

Iterating the weaker form
$p_k(s)\ge2^{(s-3)/2}p_k(s-1)$ for
$s=n-t+1,\ldots,n$ together with
\[
 \sum_{s=n-t+1}^n\frac{s-3}{2}
 =\frac{t(2n-t-5)}4
\]
proves \eqref{eq:iterated-ground-growth}.
\end{proof}

\begin{restatedtheorem}{thm-intro:connected}
For every fixed positive integer $k$, asymptotically almost every
$k$-polymatroid is connected.
\end{restatedtheorem}

\begin{proof}
The case $k=1$ is the matroid case; in fact, asymptotically almost every matroid is $3$-connected~\cite{LOSW2013}. We may therefore assume that $k\ge2$.

Let $d_k(n)$ be the number of disconnected $k$-polymatroids on
$[n]$. Every disconnected polymatroid has a smallest component of size at most $n/2$. Summing over all possible sizes $t \leq n/2$ of such a component gives
\begin{equation}\label{eq:disconnected-convolution}
 d_k(n)\le
 \sum_{t=1}^{\lfloor n/2\rfloor}
 \binom nt p_k(t)p_k(n-t).
\end{equation}
Some disconnected polymatroids may be counted more than once, which explains why we obtain an inequality rather than an equality.

Fix $\varepsilon\in(0,1)$.  By Theorem~\ref{thm:asymptotic_counting}, for
all sufficiently large $n$,
\begin{equation}\label{eq:entropy-bounds-for-connectivity}
 \lfloor k/2\rfloor (1-\varepsilon)N_n
 \le \log p_k(n)
 \le k(1+\varepsilon)N_n.
\end{equation}
Choose a fixed integer $T$ large enough that these estimates hold whenever
$n\ge T$ and that
\begin{equation}\label{eq:choose-connectivity-threshold}
 2k(1+\varepsilon)\left(\frac23\right)^T
 \le \frac{\lfloor k/2\rfloor(1-\varepsilon)}2.
\end{equation}

First suppose $1\le t<T$.  For all sufficiently large $n$,
\eqref{eq:iterated-ground-growth} gives
\[
 \frac{p_k(n-t)}{p_k(n)}\le2^{-tn/4}.
\]
Since $T$ is fixed,
\begin{equation}\label{eq:small-components-negligible}
 \sum_{t=1}^{T-1}
 \binom nt p_k(t)\frac{p_k(n-t)}{p_k(n)}
 \le
 \sum_{t=1}^{T-1}n^t p_k(t)2^{-tn/4}
 =o(1).
\end{equation}

Now let $T\le t\le n/2$.  The binomial coefficients $N_s$ are increasing,
and the elementary ratio bound $N_{s-1}/N_s\le2/3$ gives
\[
 N_t\le N_{n-t}\le\left(\frac23\right)^tN_n.
\]
Using \eqref{eq:entropy-bounds-for-connectivity} and
\eqref{eq:choose-connectivity-threshold}, we obtain, uniformly over this range of $t$,
\[
\begin{aligned}
 \log\frac{p_k(t)p_k(n-t)}{p_k(n)}
 &\le k(1+\varepsilon)(N_t+N_{n-t})
      -\lfloor k/2\rfloor(1-\varepsilon)N_n\\
 &\le-\frac{\lfloor k/2\rfloor(1-\varepsilon)}2N_n.
\end{aligned}
\]
The sum of all binomial coefficients is at most $2^n$, and hence
\begin{equation}\label{eq:large-components-negligible}
 \sum_{t=T}^{\lfloor n/2\rfloor}
 \binom nt\frac{p_k(t)p_k(n-t)}{p_k(n)}
 \le 2^{\,n-\lfloor k/2\rfloor(1-\varepsilon)N_n/2}
 =o(1).
\end{equation}

Dividing \eqref{eq:disconnected-convolution} by $p_k(n)$ and combining
\eqref{eq:small-components-negligible} with
\eqref{eq:large-components-negligible} proves
$d_k(n)/p_k(n)\to0$.
\end{proof}

\section{Representability}\label{sec:representability}

We prove Theorem~\ref{thm-intro:representable} by obtaining an upper bound on the number of $k$-polymatroids representable over at least one field; see Theorem~\ref{thm:representable}. Given a representation of a $k$-polymatroid by subspaces $V_1,\ldots,V_n$, we choose at most $k$ spanning vectors from each subspace and regard the resulting $kn$ vectors as a matroid representation. The ranks of unions of the corresponding groups of $k$ vectors recover the original polymatroid, so representable $k$-polymatroids inject into representable matroids on $kn$ elements. Nelson's bound for the latter, together with Theorem~\ref{thm:asymptotic_counting}, shows that representable $k$-polymatroids form an asymptotically negligible family.

\begin{theorem}[\cite{Nelson2018}]\label{thm:number-of-representable-matroids}
    For $n \geq 12$, the number of matroids on $[n]$ that are representable over at least one field is at most $2^{n^3/4}$.
\end{theorem}
\begin{theorem}\label{thm:representable}
    Let $\ell_k(n)$ be the number of $k$-polymatroids on $[n]$ that are representable over at least one field. Then
    \[\ell_k(n) \leq 2^{k^3 n^3/4}\]
    for sufficiently large $n$.
    In particular, almost all $k$-polymatroids are not representable over any field.
\end{theorem}
\begin{proof}
    For an $\mathbb{F}$-representable $k$-polymatroid $f$ on $[n]$, let $V_1, V_2, \dots, V_n$ be a representation of $f$ over $\mathbb{F}$.
    Since $f$ is a $k$-polymatroid, we have $\dim(V_i) \leq k$ for all $i \in [n]$. 
    Choose vectors $v_{i,1}, v_{i,2}, \dots, v_{i,k}$ in $V_i$ such that $\{v_{i,1}, v_{i,2}, \dots, v_{i, \dim V_i}\}$ is a basis of $V_i$ and $v_{i,j} = 0$ for $j > \dim V_i$.
    Let $M$ be an $\mathbb{F}$-representable matroid on $[n] \times [k]$ represented by the vectors $\{v_{i,j} \mid i \in [n], j \in [k]\}$.
    For $i \in [n]$, let $X_i$ denote the set $\{i\}\times[k]$.
    Then, for each $A\subseteq[n]$, we have
    \[
        f(A) = \dim\left(\sum_{i \in A} V_i\right) = \dim \mathrm{span}(\{v_{i,j} \mid i \in A, j \in [k]\}) = r_M\left(\bigcup_{i \in A} X_i\right). 
    \]
    Consequently, once $M$ is fixed, the formula $f(A) = r_M\left(\bigcup_{i \in A} X_i\right)$ uniquely determines $f$.
    Thus, the number of representable polymatroids is at most the number of representable matroids on $[n] \times [k]$. By Theorem~\ref{thm:number-of-representable-matroids}, the latter is at most $2^{k^3 n^3/4}$.

    Finally, for $k \geq 2$, 
    \[
        \lim_{n \to \infty} \frac{\ell_k(n)}{p_k(n)}
        \leq \lim_{n \to \infty}
        \frac{2^{k^3 n^3/4}}{2^{\lfloor k/2 \rfloor N_n (1+o(1))}}
        = 0.
    \]
\end{proof}

\begin{remark}
    A polymatroid $f$ on $[n]$ is naturally associated with a matroid $N_f$ on $X_1 \sqcup \dots \sqcup X_n$, where $|X_i| := f(\{i\})$ for $i\in [n]$.
    The matroid $N_f$ is called the \emph{multipartite matroid}~\cite{FMP2011} or the \emph{natural matroid}~\cite{BCF2023} associated with $f$.
    Farr{\`a}s, Mart{\'\i}-Farr{\'e}, and Padr{\'o}~\cite{FMP2011} showed that if $N_f$ is representable over a field $\mathbb{F}$, then $f$ is representable over $\mathbb{F}$; conversely, if $f$ is representable over $\mathbb{F}$, then $N_f$ is representable over a sufficiently large extension field $\mathbb{E}$ of $\mathbb{F}$.
    Together with Nelson's theorem, this gives another proof of Theorem~\ref{thm-intro:representable}.
\end{remark}

\begin{remark}
    Another notion of representability for polymatroids was introduced in~\cite{BHKKL2025}; it agrees with the notion used in this paper when restricted to matroids.
    Note that no proper polymatroid is representable over a field in the sense of~\cite{BHKKL2025}, by the near-idempotency principle~\cite[Prop.~4.7]{BHKKL2025} and the fact that fields are not near-idempotent.
    Therefore, by Theorem~\ref{thm-intro:proper}, almost all polymatroids are not representable over any field in the sense of~\cite{BHKKL2025}.
\end{remark}

\section*{Acknowledgments}
SI was supported by a KIAS Individual Grant (AP109501) at the Korea Institute for Advanced Study. DK was partially supported by an AMS--Simons Travel Grant.

\section*{AI Declaration}
The authors used Gemini and GPT to develop the initial idea of the proof of Theorem~\ref{thm:asymptotic_counting}.
The full proof of Theorem~\ref{thm:asymptotic_counting} was written and verified by the authors. 
In the rest of the paper, LLM-based tools are only used to polish the writing.

\providecommand{\bysame}{\leavevmode\hbox to3em{\hrulefill}\thinspace}
\providecommand{\MR}{\relax\ifhmode\unskip\space\fi MR }
\providecommand{\MRhref}[2]{%
  \href{http://www.ams.org/mathscinet-getitem?mr=#1}{#2}
}
\providecommand{\href}[2]{#2}

\end{document}